\documentclass[11pt]{article}

\usepackage[utf8]{inputenc}
\usepackage[width=16cm,height=20cm]{geometry}

\usepackage{titling}
\usepackage{cite}
\newcommand{\doi}[1]{\href{https://doi.org/#1}{doi:#1}}
\newcommand{\myurl}[1]{\href{#1}{#1}}

\usepackage{amsmath,amssymb,mathtools}

\usepackage[colorlinks, allcolors=blue]{hyperref}

\newcommand{\stirlingone}[2]{\genfrac{[}{]}{0pt}{}{#1}{#2}}

\newcommand{\eqdef}{\coloneqq}
\newcommand{\bC}{\mathbb{C}}

\newcommand{\bN}{\mathbb{N}}
\newcommand{\bNz}{\mathbb{N}_0}
\newcommand{\bQ}{\mathbb{Q}}
\newcommand{\bR}{\mathbb{R}}

\newcommand{\bZ}{\mathbb{Z}}

\newcommand{\cM}{\mathcal{M}}

\newcommand{\bF}{\mathbb{F}}
\newcommand{\cP}{\mathcal{P}}

\newcommand{\al}{\alpha}
\newcommand{\be}{\beta}
\newcommand{\ga}{\gamma}

\newcommand{\ka}{\varkappa}
\newcommand{\la}{\lambda}

\newcommand{\ph}{\varphi}
\newcommand{\si}{\sigma}
\newcommand{\Om}{\Omega}

\newcommand{\schur}{\operatorname{s}}
\renewcommand{\hom}{\operatorname{h}}

\newcommand{\Homgen}{\operatorname{H}}

\newcommand{\adj}{\operatorname{adj}}
\newcommand{\Incr}{\operatorname{Incr}}

\usepackage{amsthm}
\newtheorem{thm}{Theorem}[section]
\numberwithin{equation}{section}
\newtheorem{prop}[thm]{Proposition}
\newtheorem{lem}[thm]{Lemma}
\newtheorem{cor}[thm]{Corollary}
\theoremstyle{definition}
\newtheorem{defn}[thm]{Definition}
\newtheorem{ex}[thm]{Example}

\newtheorem{rem}[thm]{Remark}

\newcommand{\Hom}{\operatorname{H}}

\title{Complete homogeneous symmetric polynomials\\with repeating variables}

\author{Luis Angel Gonz\'{a}lez-Serrano and Egor A. Maximenko}

\begin{document}
\maketitle

\begin{abstract}
We consider polynomials of the form $\operatorname{h}_m(y_1^{[\varkappa_1]},\ldots,y_n^{[\varkappa_n]})$,
where $\operatorname{h}_m$ is the complete homogeneous polynomial of degree $m$ and $y_j^{[\varkappa_j]}$ denotes $y_j$ repeated $\varkappa_j$ times.
Using the decomposition of the generating function into partial fractions
we represent such polynomials in the form
\[
\operatorname{h}_m(y_1^{[\varkappa_1]},\ldots,y_n^{[\varkappa_n]})
=\sum_{j=1}^n \sum_{r=1}^{\varkappa_j}
\binom{r+m-1}{r-1} A_{y,\varkappa,j,r} y_j^m,
\]
where $A_{y,\varkappa,j,r}$
are some coefficients that do not depend on $m$.
We also provide an alternative proof using the inverse of the confluent Vandermonde matrix.


\end{abstract}

\bigskip
\section*{Authors' data}
Luis Angel Gonz\'{a}lez-Serrano, \quad
\myurl{https://orcid.org/0000-0001-6330-1781}, \\
e-mail: lgonzalezs@ipn.mx, lags1015@gmail.com.

\medskip\noindent
Egor A. Maximenko, \quad
\myurl{https://orcid.org/0000-0002-1497-4338}, \\
e-mail: emaximenko@ipn.mx, egormaximenko@gmail.com.

\bigskip\noindent
Escuela Superior de F\'{i}sica y Matem\'{a}ticas, Instituto Polit\'{e}cnico Nacional,
Ciudad de M\'{e}xico, Postal Code 07730,
Mexico.

\medskip
\section*{Funding}

The authors have been partially supported by Proyecto CONAHCYT ``Ciencia de Frontera''
FORDECYT-PRONACES/61517/2020,
by CONAHCYT (Mexico) scholarships,
and by IPN-SIP projects (Instituto Polit\'{e}cnico Nacional, Mexico).

\clearpage
\tableofcontents

\bigskip

\section{Introduction}

The complete homogeneous polynomials form
a generating family in the algebra of symmetric polynomials.
Therefore, they play a key role in algebraic combinatorics and some related areas,
see~\cite{Egge2019,Macdonald1995,Stanley2001}.
The complete homogeneous polynomial of degree $m$ in $N$ variables $x_1,\ldots,x_N$ is defined by the following combinatorial formula:
\begin{equation}
\label{eq:hom_def}
\hom_m(x_1,\ldots,x_N)
\eqdef
\sum_{k\in\bN_0^N\colon|k|=m}
x_1^{k_1}\cdots x_N^{k_N}.
\end{equation}
Here $\bN_0=\{0,1,2,\ldots\}$
and $|k|=k_1+\cdots+k_N$.
The sum in the right-hand side of~\eqref{eq:hom_def}
contains $\binom{N+m-1}{N}$ terms.
For large values of $N$ or $m$,
it is convenient to use the following well-known efficient formula instead of~\eqref{eq:hom_def}:
\begin{equation}
\label{eq:hom_geom_progr}
\hom_m(x_1,\ldots,x_N)
=\sum_{j=1}^N
\frac{x_j^{m+N-1}}{\prod_{k\in\{1,\ldots,N\}\setminus\{j\}}(x_j-x_k)}.
\end{equation}
This formula makes sense if $x_1,\ldots,x_N$ are pairwise different numbers.
The right-hand side of~\eqref{eq:hom_geom_progr} is also a well-defined expression in the field of rational functions over the variables $x_1,\ldots,x_N$.
Nevertheless, it is not well defined when $x_1,\ldots,x_N$ are numbers
and some of them are equal.
For example,~\eqref{eq:hom_geom_progr} cannot be applied directly to compute~$\hom_m(7,8,8)$.

It can be shown that if $x_1\ne x_2$ and $x_2=x_3$, then
\begin{equation}
\label{eq:hm_x1_x2_x2}
\hom_m(x_1,x_2,x_2)
=\frac{x_1^{m+2}}{(x_1-x_2)^2}
-\frac{x_1 x_2^{m+1}}{(x_2-x_1)^2}
+\frac{(m+1)\,x_2^{m+1}}{x_2-x_1}.
\end{equation}
Furthermore,~\eqref{eq:hm_x1_x2_x2} implies that if $x_1,x_2$ are fixed complex numbers and $|x_1|<|x_2|$, then 
$\hom_m(x_1,x_2,x_2)$ does not belong to the class $O(|x_2|^m)$ as $m$ tends to infinity,
but it belongs to the class $O(m\,|x_2|^m)$.
Here $O$ is the usual Bachmann--Landau notation
(see, e.g., \cite{Erdelyi1956}).
Estimates of this sort do not follow directly from~\eqref{eq:hom_geom_progr}.

One can obtain~\eqref{eq:hm_x1_x2_x2} from~\eqref{eq:hom_geom_progr}
by summing the fractions and using the L'Hospital's rule, but such computations become complicated in case of many variables and high multiplicities.

In this paper, we provide some generalizations of~\eqref{eq:hom_geom_progr} to compute
$\hom_m(x_1,\ldots,x_N)$
when some of $x_1,\ldots,x_N$ coincide.
We denote by $y=(y_1,\ldots,y_n)$ the list of pairwise different variables and by $\ka=(\ka_1,\ldots,\ka_n)$ the list of multiplicities.
Then, we use the following short notation for the list of variables with repetitions:
\[
y^{[\ka]}
=
\Bigl(\,\underbrace{y_1,\ldots,y_1}_{\ka_1\ \text{times}}\,\,\ldots,\,
\underbrace{y_n,\ldots,y_n}_{\ka_n\ \text{times}}\,\Bigr).
\]
Our main results
(Theorems~\ref{thm:hom_rep_via_powers_A}
and~\ref{thm:hom_rep_via_powers_B})
are expansions of $\hom_m(y^{[\ka]})$ of the form
\begin{equation}
\label{eq:general_form_of_expansion}
\hom_m(y^{[\ka]})
=
\sum_{j=1}^n P_{y,\ka,j}(m) y_j^m,
\end{equation}
where $P_{y,\ka,j}(m)$ are some polynomials in variable $m$
whose coefficients are explicitly given in terms of $y_1,\ldots,y_m$ and $\ka_1,\ldots,\ka_m$.
We propose two (different but equivalent) formulas of the form~\eqref{eq:general_form_of_expansion}.

The paper has the following structure.
Section~\ref{sec:repeating}
introduces a notation for repeating variables.
Section~\ref{sec:multisets} explains some notation from combinatorics and explores the idea of identifying (``grouping together'')
elements of multisets.
In Section~\ref{sec:basic_facts_hom}, we recall some basic facts about complete homogeneous polynomials.
In Section~\ref{sec:combin_formula_hom_rep}, we present an elementary combinatorial formula that follows from~\eqref{eq:hom_def} by grouping together the coinciding monomials.

In Section~\ref{sec:bialternant_Schur_rep},
we recall a formula that expresses
every Schur polynomial with repeated variables
as a quotient of the determinants of generalized confluent Vandermonde matrices;
this formula was recently proven in~\cite{AS2024} and~\cite{GM2023}.

In Section~\ref{sec:def_AB},
we define expressions
$A_{y,\ka,s,r}$ and $B_{y,\ka,s,r}$.
They are rational functions in variables $y_1,\ldots,y_n$;
they also can be written
in terms of complete homogeneous polynomials
in variables $y_d/(y_d-y_s)$
and $1/(y_d-y_s)$, respectively
($d\in\{1,\ldots,n\}\setminus\{s\}$).
The expressions $A_{y,\ka,s,r}$
and $B_{y,\ka,s,r}$
play a crucial role in the subsequent sections.

In Section~\ref{sec:rational_decomposition},
we prove the following expansions of rational functions (with unital numerator)
into elementary fractions:
\begin{align}
\label{eq:rational_expansion_A}
\frac{1}{(1-y_1 t)^{\ka_1}\cdots (1-y_n t)^{\ka_n}}
&=
\sum_{s=1}^n \sum_{r=1}^{\ka_r}
\frac{A_{y,\ka,s,r}}{(1-y_s t)^r},
\\[1ex]
\label{eq:rational_expansion_B}
\frac{1}{(t-y_1)^{\ka_1}\cdots (t-y_n)^{\ka_n}}
&=\sum_{s=1}^n \sum_{r=1}^{\ka_r}
\frac{B_{y,\ka,s,r}}{(t-y_s)^r},
\end{align}
In Section~\ref{sec:hom_rep_expansion},
we prove our main results, i.e.,
two expansions of the form~\eqref{eq:general_form_of_expansion}:
\begin{align}
\label{eq:hom_expansion_A}
\hom_m(y^{[\ka]})
&=
\sum_{s=1}^n \sum_{r=1}^{\ka_s}
\binom{r+m-1}{r-1}\;
A_{y,\ka,s,r}\;
y_s^m,
\\[1ex]
\label{eq:hom_expansion_B}
\hom_m(y^{[\ka]})
&=
\sum_{s=1}^{n}
\sum_{r=1}^{\ka_s} 
\binom{m + |\ka| - 1}{r - 1}\;
B_{y,\ka,s,r}\;
y_s^{m + |\ka| -r}.
\end{align}
Our proof of~\eqref{eq:hom_expansion_A}
is based on~\eqref{eq:rational_expansion_A}.
To prove~\eqref{eq:hom_expansion_B},
we use the formula from Section~\ref{sec:bialternant_Schur_rep}
and the formula for the inverse of the confluent Vandermonde matrix~\cite{MoucoufZriaa2022}.
We also show the equivalence of~\eqref{eq:hom_expansion_A} and~\eqref{eq:hom_expansion_B}
directly.

We hope that this paper will be useful in studying some properties of the determinants and minors of banded Toeplitz matrices
generated by Laurent polynomials with multiple zeros.
Some connections between banded Toeplitz matrices and skew Schur polynomials were established in
\cite{A2012,
AGMM2019,
BumpDiaconis2002,
GarciaGarciaTierz2020,
MM2017}.

This paper may also be useful for better understanding the structure of the inverse of the confluent Vandermonde matrix~\cite{HouPang2002,
MoucoufZriaa2022,
Respondek2011,
Respondek2024}
and to obtain upper bounds for
the values of Schur polynomials
with repeating variables.

\medskip

\section{Notation for repeating variables}
\label{sec:repeating}

In this section, we recall some constructions from~\cite{GM2023}.
Suppose that $n\in\bN$.
Given $\ka = (\ka_1, \ldots, \ka_{n})$ in $\bNz^n$,
we denote by $|\ka|$ the sum of the elements of $\ka$:
\[
|\ka|\eqdef\sum_{j=1}^{n}\ka_j,
\]
and by $\si(\ka)$ be the list of the partial sums of $\ka$:
\begin{equation}
\label{eq:partial_sums_def}
\si(\ka)_q \eqdef \sum_{j=1}^{q} \ka_j\qquad
(q\in \{0, \ldots, n\}).
\end{equation}
In particular, $\si(\ka)_0\eqdef0$ for each $\ka$.

\begin{ex}
Let $\ka = (6, 4, 7, 1)$. Then
\[
\si(\ka)_0 = 0, \qquad
\si(\ka)_1 = 6, \qquad
\si(\ka)_2 = 10, \qquad
\si(\ka)_3 = 17, \qquad
\si(\ka)_4 = |\ka| = 18.
\]
\end{ex}

\begin{defn}
\label{def:gamma}
Let $\ka \in \bN^n$, $N \eqdef |\ka|$. 
Consider the following set of pairs of natural numbers:
\[
Q_\ka
\eqdef
\bigl\{(q,r)\in\bN^2\colon\quad
q\le n,\ r\le \ka_q\bigr\}.
\]
Define
$\rho_\ka\colon Q_\ka\to\{1,\ldots,N\}$
by
\[
\rho_\ka(q,r) \eqdef \si(\ka)_{q-1}+r.
\]
Define $\ga_\ka\colon\{1,\ldots,N\}\to Q_\ka$ by $\ga_\ka(k)\eqdef (q,r)$,
where
\begin{equation}
\label{eq:qr_def}
q \eqdef
\max \bigl\{j\in \{1,\ldots, n\} \colon \si(\ka)_{j-1} < k\bigr\},
\qquad 
r \eqdef k - \si(\ka)_{q-1}. 
\end{equation}
\end{defn}

In the context of~\eqref{eq:qr_def},
\begin{equation}
\label{eq:between_partial_sums}
\si(\ka)_{q-1} < k \leq \si(\ka)_{q}.
\end{equation}

\begin{rem}[consistency of the definitions of $\rho_\ka$ and $\ga_\ka$]
\label{rem:rho_and_ga_are_consistent}
Let $\ka\in\bN^n$.
Let us verify that the definitions of $\ga_\ka$ and $\rho_\ka$ are consistent,
i.e.,
the values of $\ga_\ka$ and $\rho_\ka$ belong to the indicated codomains.
If $(q,r)\in Q_\ka$, then
\[
\rho_\ka(q,r)
=\si(\ka)_{q-1}+r
\le \si(\ka)_{q-1} + \ka_q
= N.
\]
If $k\in\{1,\ldots,N\}$
and $(q,r)=\ga_\ka(k)$,
then $q\in\{1,\ldots,n\}$,
and from inequalities~\eqref{eq:between_partial_sums}
we get
\[
r = k - \si(\ka)_{q-1}
\le \si(\ka)_q - \si(\ka)_{q-1}
= \ka_q.
\]
\qedhere
\end{rem}

\begin{ex}
Let $\ka = (6,4,7,1)$.
Then $\ga_\ka(14) = (3,4)$, because 
    \[
    \underbrace{10}_{\si(\ka)_2} < 14 = \underbrace{10}_{\si(\ka)_2} + \underbrace{4}_{r} \leq \underbrace{17}_{\si(\ka)_3}.
    \]
\hfill\qed
\end{ex}

\begin{ex}
Let $\ka=(3, 2, 5)$.
Then $N=|\ka|=10$,
\[
Q_\ka =
\bigl\{
(1,1),\ (1,2),\ (1,3) ,\  
(2,1),\ (2,2),\   
(3,1),\ (3,2),\ (3,3),\ (3,4),\ (3,5) 
\bigr\},
\]
and
\begin{align*}
&
\ga_\ka(1) = (1, 1), \quad
\ga_\ka(2) = (1, 2), \quad 
\ga_\ka(3) = (1, 3),
\\
&
\ga_\ka(4) = (2, 1), \quad
\ga_\ka(5) = (2, 2),
\\
&
\ga_\ka(6) = (3, 1), \quad 
\ga_\ka(7) = (3, 2), \quad
\ga_\ka(8) = (3, 3), \quad
\ga_\ka(9) = (3, 4), \quad
\ga_\ka(10) = (3, 5).
\end{align*}
\hfill\qed
\end{ex}

\begin{prop}
\label{prop:ga_rho_inverse}
Let $\ka\in\bN^n$.
Then $\ga_\ka$ is the inverse function
to $\rho_\ka$.
Moreover, $\ga_\ka$ is the enumeration of $Q_\ka$ in the lexicographic order.
\end{prop}

In other words,
the first part of Proposition~\ref{prop:ga_rho_inverse}
means that for every
$k$ in $\{1,\ldots,N\}$,
$\ga_\ka(k)$
is the unique pair $(q,r)$ such that
$q\in \{1, \ldots, n\}$,
$r\in \{1, \ldots, \ka_q\}$,
and
\begin{equation}
\label{eq: partial_sums_residue}
k = \si(\ka)_{q-1} + r.
\end{equation}

\begin{defn}
\label{def:y_rep_ka}
Let $\ka\in\bNz^n$ and $y_1,\ldots,y_{n}$ be some numbers or variables.
We denote by $y^{[\ka]}$ the list of length $N \eqdef |\ka|$ constructed in the following manner.
For every $k$ in $\{1, \ldots, N\}$,
we put $(q,r)\eqdef\ga_\ka(k)$ and
define the $k$th component of $y^{[\ka]}$ to be $y_q$:
\[
(y^{[\ka]})_k \eqdef y_q.
\]
\end{defn}

\begin{ex}
Let $y=(y_1,y_2,y_3, y_4)$ and $\ka=(4,1,5)$.
Then
\[
y^{[\ka]}=(y_1,y_1,y_1,y_1,\ y_2,\ y_3,y_3,y_3,y_3,y_3).
\]
\hfill\qed
\end{ex}

\begin{rem}[multivariate polynomials]
\label{rem:polynomials}
Let $\bF$ be a field of characteristic $0$.
For example, $\bF$ can be
the field of rational numbers $\bQ$,
or the field of real numbers $\bR$,
or the field of complex numbers $\bC$.
We treat polynomials in one or many variables
as lists of coefficients;
see detail in
Hungerford~\cite[Chapter~3, Section~4]{Hungerford1980},
Lang~\cite[Chapter~4, Sections~1, 7]{Lang2005},
and~\cite{Kures2010}.
We use the traditional notation
$\bF[x_1,\ldots,x_N]$
for the ring of polynomials in $N$ variables
and $\bF(x_1,\ldots,x_N)$
for the corresponding quotient field
(whose elements are called ``rational functions'').
Notice that the names of the variables are not important.
It is well known (see a proof in~\cite{Kures2010})
that the evaluation of polynomials is a homomorphism.
\end{rem}

\begin{rem}[the repetition homomorphism]
\label{rem:rep_homomorphism}
Let $\ka\in\bN^n$ and $N\eqdef|\ka|$.
We define
\[
\Phi_\ka\colon
\bF[x_1,\ldots,x_N]
\to\bF[y_1,\ldots,y_n]
\]
by
\[
(\Phi_\ka(f))(y)\eqdef f(y^{[\ka]}),
\qquad\text{i.e.},\qquad
(\Phi_\ka(f))(y_1,\ldots,y_n)
\eqdef
f(y_1^{[\ka_1]},\ldots,y_n^{[\ka_n]}).
\]
$\Phi_\ka$
is a homomorphism
that preserves the identity element.
See Kure\v{s}~\cite{Kures2010}
for a much more general explanation of the polynomial homomorphisms defined by compositions.
With this notation,
$\hom_\la(y^{[\ka]})$
is the image of $\hom_\la(x)$ under $\Phi_\ka$.
\end{rem}

\medskip

\section{Identifying elements of multisets}
\label{sec:multisets}

The ideas of this section will be used in the proof of Proposition~\ref{prop:hom_with_rep_combinatorial}. 
In this section, we suppose that $n\in\bN$ and $m\in\bNz$.

Let $\cM_{n,m}$ be the set of all $n$-tuples of nonnegative integers with sum equal to $m$:
\begin{equation}
\label{eq:D_def}
\cM_{n,m}
\eqdef
\bigl\{\ka\in\bNz^n\colon\ |\ka|=m\bigr\}.
\end{equation}
Some authors denote this set by
$\bigl(\!\binom{[n]}{m}\!\bigr)$.
The elements of $\cM_{n,m}$ can be identified with multisets of size $m$ whose elements belong to $\{1,\ldots,n\}$,
or, in other words, with
``weak compositions of $m$ into $n$ parts'' \cite[Section 1.2]{Stanley2011}, or
``$n$-combinations of $m$ different objects, each available in unlimited supply'' \cite[Section 2.5]{Brualdi2009}.

Furthermore,
let $\Incr_{n,m}$ be the set of all nondecreasing functions $\tau\colon\{0,1,\ldots,n\}\to\bNz$ such that $\tau_0=0$ and $\tau_n=m$:
\begin{equation}
\label{eq:Incr_def}
\Incr_{n,m}\eqdef
\bigl\{
\tau\in\bN_0^n\colon\quad
0=\tau_0\le\tau_1\le\ldots\le\tau_n=m\bigr\}.
\end{equation}
For each $\ka$ in $\cM_{n,m}$,
let $\si(\ka)$ be the list of the partial sums of $\ka$ defined by \eqref{eq:partial_sums_def}. The following proposition is a well-known fact from elementary combinatorics \cite[Theorem 2.5.1]{Brualdi2009}.

\begin{prop}
\label{prop:multisets_to_incr_bijection}
$\si\colon\cM_{n,m}\to\Incr_{n,m}$ is a bijection.
The number of elements in $\cM_{n,m}$ equals
\[
\#\cM_{n,m}
=\binom{n+m-1}{n-1}
=\binom{n+m-1}{m}.
\]
\end{prop}

\begin{ex}
For $n=3$ and $m=2$,
\[
\cM_{3,2}
=\bigl\{(0, 0, 2),\ (0, 1, 1),\ (0, 2, 0),\ 
(1, 0, 1),\ (1, 1, 0),\ (2, 0, 0)\bigr\},
\]
and $\binom{n+m-1}{n-1}=\binom{4}{2}=6$.
\hfill\qed
\end{ex}

The forthcoming Definition~\ref{def:grouping_function} and Proposition~\ref{prop:group_indices} correspond to the situation when we group together and identify some elements of multisets.
For the purpose of this paper, it is sufficient to identify some adjacent elements of $\{1,\ldots, N\}$, but one could generalize this idea to general set partitions (or equivalence relations).

\begin{defn}
\label{def:grouping_function}
Let $\ka\in\bN_0^n$, $N \eqdef |\ka|$.
Define $\ph_\ka\colon\bN_0^{N}\to\bN_0^n$,
\begin{equation}
\label{eq:ph_def}
\ph_\ka(\be)_q
\eqdef \sum_{\si(\ka)_{q-1}<p\le\si(\ka)_q}
\be_p
\qquad(\beta\in\bN_0^{N},\quad q\in\{1,\ldots,n\}).
\end{equation}
\end{defn}

\begin{ex}
Let $\ka=(2,3)$ and $\al=(1,2)$. Then $\si(\ka)_0 = 0$, $\si(\ka)_1= 2$, $\si(\ka)_2 = 5$, and  $\ph_\ka^{-1}(\{\al\})$ consists of the following elements (the underbrackets just illustrate the grouping corresponding to $\ka$):
\begin{align*}
&(\underbracket{1,0}, \underbracket{2,0,0}),
(\underbracket{1,0}, \underbracket{0,2,0}),
(\underbracket{1,0}, \underbracket{0,0,2}),
(\underbracket{1,0}, \underbracket{1,1,0}),
(\underbracket{1,0}, \underbracket{1,0,1}),
(\underbracket{1,0}, \underbracket{0,1,1}),
\\[1ex]
&(\underbracket{0,1}, \underbracket{2,0,0}),
(\underbracket{0,1}, \underbracket{0,2,0}),
(\underbracket{0,1}, \underbracket{0,0,2}),
(\underbracket{0,1}, \underbracket{1,1,0}),
(\underbracket{0,1}, \underbracket{1,0,1}),
(\underbracket{0,1},
\underbracket{0,1,1}).
\end{align*}
The cardinality of $\ph_\ka^{-1}(\{\al\})$ is $12$.
\hfill\qed
\end{ex}

\begin{prop}
\label{prop:group_indices}
Let $\ka\in\bN_0^n$, $N\eqdef |\ka|$, and $m\in\bN_0$.
Then
\begin{equation}
\label{eq:image_of_phi_which_sums_the_indices}
\bigl\{\ph_\ka(\beta)\colon\ \beta\in\cM_{N,m}\bigr\}
=\cM_{n,m}.
\end{equation}
Moreover, for every $\al$ in $\cM_{n,m}$,
the inverse image of $\{\al\}$ with respect to $\ph$ has cardinality
\begin{equation}
\label{eq:number_of_j_such_that_phi_j_eq_k}
\#\Bigl(\ph_\ka^{-1}(\{\al\})\Bigr)
=\prod_{p=1}^n
\binom{\ka_p+\al_p-1}{\al_p}.
\end{equation}
\end{prop}

\begin{proof}
The definition of $\ph_\ka$ implies
that $\ph_\ka(\beta)\in\cM_{n,m}$ for every $\beta$ in $\cM_{N,m}$.
Given $\al$ in $\cM_{n,m}$,
\[
\ph_\ka^{-1}(\{\al\})
=
\bigl\{\beta\in\bN_0^{N}\colon\ \ph_\ka(\beta)=\al\bigr\}
=
\prod_{p=1}^n \cM_{\ka_p,\al_p}.
\]
The cardinality of the cartesian product on the right-hand side
is given by the right-hand side of~\eqref{eq:number_of_j_such_that_phi_j_eq_k}.
In particular, this cartesian product is nonempty.
\end{proof}

\medskip

\section{Basic facts about complete homogeneous polynomials}
\label{sec:basic_facts_hom}

In this section, we recall some well-known facts about complete homogeneos polynomials;
see also
\cite[Section~2.2]{Egge2019},
\cite[Section~1.2]{Macdonald1995},
and~\cite[Section~7.5]{Stanley2001}.
We will use the notion of formal series
and quotients of formal series;
see, e.g.,
\cite[Chapter~3, Section~5]{Hungerford1980}.

The combinatorial definition~\eqref{eq:hom_def}
easily implies the following recurrent formula:
\begin{equation}
\label{eq:hom_recur}
\hom_{m+1}(x_1,\ldots,x_N,x_{N+1})
=\hom_{m+1}(x_1,\ldots,x_N)
+x_{N+1}\hom_m(x_1,\ldots,x_{N+1}).
\end{equation}
In what follows, we abbreviate the list $(x_1,\ldots,x_N)$ by $x$.
So, \eqref{eq:hom_recur} can be written as
\[
\hom_{m+1}(x, x_{N+1})
= \hom_{m+1}(x) + x_{N+1} \hom_m(x).
\]
The generating function of the sequence $(\hom_m(x))_{m=0}^\infty$ is defined as
the following formal series:
\[
\Homgen(x)(t)\eqdef
\sum_{m=0}^\infty\hom_m(x)t^m.
\]
Notice that $\Homgen(x)(t)$ is an element of $R[[t]]$, where $R=\bF[x_1,\ldots,x_N]$.

The recurrent formula~\eqref{eq:hom_recur}
implies the following recurrent formula
for the generating functions:
\begin{equation}
\label{eq:Homgen_recur}
(1-x_{N+1}t)\Homgen(x,x_{N+1})(t)=\Homgen(x)(t).
\end{equation}
By induction over $N$,
we get the following explicit formula
for the generating function:
\begin{equation}
\label{eq:Homgen_explicit}
\Hom(x)(t)
=\frac{1}{\prod_{j=1}^N (1-x_j t)}.
\end{equation}
Notice that~\eqref{eq:Homgen_explicit}
is an identity in $R((t))$,
where $R=\bF[x_1,\ldots,x_N]$.

The fraction from the right-hand side of~\eqref{eq:Homgen_explicit}
can be decomposed into simple fractions:
\begin{equation}
\label{eq:decomp_simple_fraction1}
\frac{1}{\prod_{j=1}^N (1-x_j t)}
= \sum_{s=1}^{N}
\frac{x_s^{N-1}}{\displaystyle\prod_{\substack{1 \le j\le N\\j\ne s}}(x_s - x_j)}\,
\frac{1}{1-x_s t}.
\end{equation}
Furthermore,
$(1-y_s t)^{-1}
=\sum_{m=0}^\infty y_s^m t^m$.
Therefore,
\[
\sum_{m=0}^\infty \hom_m(x)\,t^m
=\sum_{m=0}^\infty
\Biggl(\sum_{s=1}^{N}
\frac{x_s^{m+N-1}}{\displaystyle\prod_{\substack{1 \le j\le N\\j\ne s}}(x_s - x_j)} \Biggr)\,t^m.
\]
Equating the coefficients,
we obtain~\eqref{eq:hom_geom_progr}.

The following fact is a particular case of Macdonald~\cite[Chapter I, Section 3, Example 10, page 47]{Macdonald1995} and Egge~\cite[Problem 3.11]{Egge2019}.
It was also proven by Gomezllata Marmolejo~\cite[Lemma~3.5]{Gomezllata2022}.

\begin{prop}
\label{prop:Gomezllata}
Let $m\in\bNz$.
The following identity holds
in $\bF[x_1,\ldots,x_n]$:
\begin{equation}
\label{eq:Gomezllata}
\hom_m(1 + x_1, \ldots, 1 + x_n) = \sum_{k = 0}^m \binom{m + n - 1}{m - k}\hom_{k}(x_1,\ldots,x_n)
\end{equation}
In the short notation,
\[
\hom_m(1^{[n]} + x) = \sum_{k = 0}^m \binom{m + n - 1}{m - k}\hom_{k}(x).
\]
\end{prop}

\begin{proof}
Gomezllata Marmolejo proved this identity by combinatorial arguments.
We give another proof based on \eqref{eq:hom_geom_progr}:
\begin{align*}
    \hom_m(1^{[n]} + x) &= \sum_{j=1}^n \frac{(1 + x_j)^{m + n - 1}}{\prod_{k \neq j}((1 + x_j) - (1 + x_k))} 
    =
    \sum_{k = 0}^{m + n - 1} \binom{m + n - 1}{k}\sum_{j=1}^n\frac{x_j^{k}}{\prod_{k \neq j}(x_j - x_k)} 
    \\[1ex]
    &= 
    \sum_{k = n - 1}^{m + n - 1} \binom{m + n - 1}{k} \hom_{k - n + 1}(x)
    =
    \sum_{k = 0}^{m} \binom{m + n - 1}{m - k} \hom_{k}(x).
\end{align*}
In the third equality, we have used the well-known fact that the inner sum vanishes for $k< n - 1$ (see, for example, \cite[Lemma 4.3]{AGMM2019}).
\end{proof}

\medskip

\section{\texorpdfstring{Combinatorial formula for $\boldsymbol{\hom_m(y^{[\ka]})}$}{Combinatorial formula for hm(ykappa)}}
\label{sec:combin_formula_hom_rep}

For $x=y^{[\ka]}$,
some monomials in the right-hand side of~\eqref{eq:hom_def} coincide.
Grouping them together,
we obtain the following combinatorial identity for $\hom_m(y^{[\ka]})$.

\begin{prop}
\label{prop:hom_with_rep_combinatorial}
Let $\ka\in\bN_0^n$ and $m\in\bN_0$.
Then the following identity holds
in $\bF[y_1,\ldots,y_n]$:
\begin{equation}
\label{eq:hom_with_rep_combinatorial}
\hom_m(y^{[\ka]})
=
\sum_{k\in\cM_{n,m}}
\prod_{j=1}^n
\binom{\ka_j+k_j-1}{k_j}
\prod_{j=1}^n y_j^{k_j}.
\end{equation}
\end{prop}

\begin{proof}
We present a detailed proof based on Proposition~\ref{prop:group_indices}.
Let $N \eqdef |\ka|$ and $\ph_\ka$ be as in Proposition~\ref{prop:group_indices}.
Due to~\eqref{eq:image_of_phi_which_sums_the_indices},
$\cM_{N,m}$ decomposes into the following disjoint union:
\begin{equation}
\label{eq:setj_separation}
\cM_{N,m}
=\bigcup_{k\in\cM_{n,m}}
\bigl\{j\in\bNz^{N}\colon\ \ph_\ka(j)=k\bigr\}.
\end{equation}
We apply~\eqref{eq:hom_def}
with $x\eqdef y^{[\ka]}$
and separate the set of the indices according to~\eqref{eq:setj_separation}.
In other words,
we group together the summands corresponding to the indices $j$ having the same image $k$ with respect to $\ph_\ka$:
\[
\hom_m(y^{[\ka]})
=\hom_m(x)
=\sum_{j\in\cM_{N,m}}
\prod_{p=1}^{N}x_p^{j_p}
=\sum_{k\in\cM_{n,m}}
\sum_{\substack{j\in\bN_0^{N}\\\ph_\ka(j)=k}}
\prod_{p=1}^{N}x_p^{j_p}.
\]
When $\ph_\ka(j)=k$,
\[
\prod_{p=1}^{N}x_p^{j_p}
=\prod_{q=1}^n
\prod_{p = \si_\ka(q-1) + 1}^{\si_\ka(q)}
x_p^{j_p}
=\prod_{q=1}^n
\prod_{p = \si_\ka(q-1) + 1}^ {\si_\ka(q)}
y_q^{j_p}
=\prod_{q=1}^n y_q^{k_q}.
\]
So,
\[
\hom_m(y^{[\ka]})
=\sum_{k\in\cM_{n,m}}
\sum_{\substack{j\in\bN_0^{N}\\\ph_\ka(j)=k}}
\prod_{q=1}^n y_q^{k_q}.
\]
For a fixed $k$ in $\bN_0^n$, all summands in the inner sum (over $j$) are equal to each other.
Formula~\eqref{eq:number_of_j_such_that_phi_j_eq_k} yields the number of such summands, and we obtain~\eqref{eq:hom_with_rep_combinatorial}.
\end{proof}

As a particular case of Proposition~\ref{prop:hom_with_rep_combinatorial}, we compute $\hom_m$ of one variable $t$ repeated $r$ times, i.e.,
$\hom_m\Bigl(\,\underbrace{t,\ldots,t}_{r}\,\Bigr)$.
The trivial case $m<0$
will be useful in the following sections.

\begin{prop}
\label{prop:hom_one_variable_with_repetitions}
Let $r\in\bN$, $m\in\bZ$.
Then the following identity holds in $\bF[t]$:
\begin{equation}
\label{eq:hom_one_variable_with_repetitions}
\hom_m(t^{[r]})
=
\binom{r+m-1}{r - 1}
\hom_m(t)
=
\begin{cases}
\binom{r+m-1}{r - 1} t^m, & m\ge0;
\\[1ex]
0, & m<0.
\end{cases}
\end{equation}
\end{prop}

\begin{proof}
Apply Proposition \ref{prop:hom_with_rep_combinatorial} to one variable.
\end{proof}

\begin{cor}
Let $\ka\in\bN_0^n$ and $m\in\bN_0$.
Then the following identity holds in $\bF[y_1,\ldots,y_m]$:
\begin{equation}
\label{eq:hom_with_rep_combinatorial_cor}
\hom_m(y^{[\ka]})
=
\sum_{k\in\cM_{n,m}}
\prod_{j=1}^n
\hom_{k_j}(y_j^{[\ka_j]}).
\end{equation}
\end{cor}

\begin{proof}
It follows from Propositions~\ref{prop:hom_with_rep_combinatorial} and \ref{prop:hom_one_variable_with_repetitions}.
\end{proof}

\begin{rem}[the derivative of a univariate polynomial]
\label{rem:derivative_of_polynomial}
We recall that the derivative of a polynomial
can be defined in a purely algebraic way
\cite[Chapter IV, Section 3]{Lang2005}. 
Namely, if $f\in \bF[t]$ is given by the list of coefficients $(f_0, f_1, \ldots,f_m)$, then $f'$ is given by the list of coefficients $(f_1, 2f_2, \ldots, mf_m)$. 
This rule naturally extends to quotients of polynomials: $(f / g)' = (f'g - fg')/g^2$. 
\end{rem}

The polynomial $\hom_m(t^{[r]})$ from Proposition~\ref{prop:hom_one_variable_with_repetitions}
can also be expressed through the $(r-1)$th derivative of the monomial $\hom_{m+r-1}(t)$.

\begin{cor}
\label{cor:hom_rep_one_variable_through_derivative}
Let $r\in\bN$, $m\in\bZ$.
Then the following identity holds in $\bF[t]$:
\begin{equation}
\label{eq:hom_one_variable_rep_via_derivative}
\hom_m(t^{[r]})
=
\frac{1}{(r-1)!} \hom_{m + r-1}^{(r-1)}(t).
\end{equation}
\end{cor}

\medskip

\section{Bialternant formula for Schur polynomials with repeating variables}
\label{sec:bialternant_Schur_rep}

Given an integer partition $\la$,
we denote by $\schur_\la(x_1,\ldots,x_n)$ or shortly by $\schur_\la(x)$ the Schur polynomial associated to $\la$,
in the variables $x_1,\ldots,x_n$.
Basic facts about Schur polynomials can be found in \cite[Section 1.3]{Macdonald1995} and \cite[Section 7.15]{Stanley2001}.
In particular, $\schur_\la(x)$ can be written as a quotient of two determinants, see~\cite{GM2023}:

\begin{equation}
\label{eq:schur_bialternant}
\schur_\la(x)
=
\frac{\det \bigl[ x_k^{j - 1 + \la_{N + 1 - j}} \bigr]_{j,k = 1}^{N}}%
{\det \bigl[ x_k^{j - 1} \bigr]_{j,k = 1}^{N}}.
\end{equation}

In this section, we explain a generalization of~\eqref{eq:schur_bialternant} to the case if some of $x_1,\ldots,x_n$ coincide
(Theorem~\ref{thm:bialternant_formula_Schur_rep}),
and we apply it to the particular case $\la=(m)$ corresponding to complete homogeneous polynomials.

\begin{defn}[generalized confluent Vandermonde matrix associated to an integer partition,
a list of multiplicities,
and a list of variables]
\label{defn:G}
Let $\ka\in\bN^n$,
$N\eqdef|\ka|$,
$\la = (\la_1,\ldots, \la_{N}) \in \cP(N)$,
and $y=(y_1,\ldots,y_n)$
be a list of variables or numbers.
We denote by $G_{\la,\ka}(y)$
the $N\times N$ matrix with the following components.
Given $j,k$ in $\{1,\ldots, N\}$,
we put $(q, r) \eqdef \ga_\ka(k)$ and
\begin{equation}
\label{eq:G_def}
G_{\la,\ka}(y)_{j,k}
\eqdef 
\begin{cases}
\displaystyle
\binom{j - 1 + \la_{N + 1 - j}}{r - 1}
y_q^{j - r + \la_{N + 1 - j}},
& \text{if}\  j - r + \la_{N + 1 - j} 
\ge 0;
\\[1.5ex]
0, & \text{otherwise}.
\end{cases}
\end{equation}
Here, as always,
$\la_p=0$ if $p>\ell(\la)$.
\end{defn}

Using Proposition~\ref{prop:hom_one_variable_with_repetitions} and Corollary~\ref{eq:hom_one_variable_rep_via_derivative}
we get other equivalent forms of~\eqref{eq:G_def}:
\begin{align}
\label{eq:G_def_hom}
G_{\la,\ka}(y)_{j,k}
&=
\binom{j - 1 + \la_{N + 1 - j}}{r - 1}
\hom_{j - r + \la_{N  + 1 - j}}(y_{q}),
\\[1ex]
\label{eq:G_def_hom_rep}
G_{\la,\ka}(y)_{j,k}
&=
\hom_{j - r  + \la_{N + 1 - j}}(y_{q}^{[r]}),
\\[1ex]
\label{eq:G_def_deriv}
G_{\la,\ka}(y)_{j,k}
&=
\frac{1}{(r-1)!}\,
\hom_{j - 1 + \la_{N + 1 - j}}^{(r - 1)}(y_{q}).
\end{align}

\begin{rem}[confluent Vandermonde matrix]
\label{rem:confluent_Vandermonde_matrix}
For $\la$ being the empty partition $\emptyset$,
we use short notation $V_\ka(y)$
instead of $G_{\emptyset,\ka}(y)$.
$V_\ka(y)$ is called the
\emph{confluent Vandermonde matrix}
associated to the points $y_1,\ldots,y_n$
and multiplicities $\ka_1,\ldots,\ka_n$.
The following formula for its determinant
is well known:
\begin{equation}
\label{eq:det_confluent_Vandermonde}
\det V_{\ka}(y)
= 
\prod_{1\le j<k\le n}
(y_k - y_j)^{\ka_j \ka_k}.
\end{equation}
Some proofs of~\eqref{eq:det_confluent_Vandermonde}
were given by Ha, Gibson~\cite{HaGibson1980}
and Kalman~\cite{Kalman1984};
see also a historical review by Respondek~\cite{Respondek2024}.
\end{rem}

The following theorem was proven
by Averkov and Scheiderer~\cite{AS2024};
see also more detailed proofs in our paper~\cite{GM2023}.

\begin{thm}
\label{thm:bialternant_formula_Schur_rep}
Let $\ka\in\bNz^n$, $N\eqdef|\ka|$, and $\la$ be a partition of length $\le N$.
Then the following identity holds
in $\bF[y_1,\ldots,y_n]$:
\begin{equation}
\label{eq:bialternant_formula_Schur_rep}
\schur_\la\bigl(y^{[\ka]}\bigr)
=\frac{\det G_{\la, \ka}(y)}{\det V_{\ka}(y)}.
\end{equation}
\end{thm}

Applying Theorem~\ref{thm:bialternant_formula_Schur_rep}
to $\la=(m)$
we obtain the following particular case.

\begin{cor}[complete homogeneous polynomial with repeated variables as a quotient of two determinants]
\label{cor:hom_rep_bialternant}
Let $\ka\in\bNz^n$ and $m\in\bNz$.
Then the following identity holds
in $\bF[y_1,\ldots,y_n]$:
\begin{equation}
\label{eq:hom_rep_bialternant}
\hom_{m}(y^{[\ka]})
= \frac{\det G_{(m), \ka}(y)}{\det V_{\ka}(y)}.
\end{equation}
\end{cor}

\begin{rem}
\label{rem:G_m}
For $\la=(m)$, the rule from Definition~\ref{defn:G}
simplifies as follows:
\begin{align*}
    G_{(m), \ka}(y)_{j,k} 
    =
    \begin{cases}
    \displaystyle \binom{N - 1 + m}{r - 1}
    y_q^{m+N-r}, &\text{if} \quad j = N;
    \\[2ex]
    \displaystyle \binom{j - 1}{r - 1}\hom_{j - r}(y_q), &\text{if} \quad 1\leq j \le N-1.
    \end{cases}        
    \end{align*}
\end{rem}

\begin{ex}
Let $n = 2$, $\ka = (1,2)$, and $m\in \bNz$.
Then
\begin{align*}
    \hom_m(y^{[\ka]}) 
    &=
    \frac{
    \det\begin{bmatrix}
    \binom{0}{0} & \binom{0}{0} & 0 \\[1ex]
    \binom{1}{0}y_1 & \binom{1}{0}y_2 & \binom{1}{1}\\[1ex]
\binom{m + 2}{0} y_1^{m + 2} & \binom{m + 2}{0} y_2^{m + 2}  & \binom{m + 2}{1} y_2^{m + 1}
    \end{bmatrix}
    }{
    \det\begin{bmatrix}
    \binom{0}{0} & \binom{0}{0} & 0 \\[1ex]
    \binom{1}{0}y_1 & \binom{1}{0}y_2 & \binom{1}{1}\\[1ex]
    \binom{2}{0}y_1^{2} & \binom{2}{0}y_2^{2} & \binom{2}{1}y_2
    \end{bmatrix}
    }
\\[1ex]
&=
\frac{
 y_1^{m + 2} - y_2^{m + 2} + (m + 2)y_2^{m + 1}(y_2 - y_1)}{(y_1 - y_2)^2}
\\[1ex]
&=
\frac{
y_1^{m + 2} - (m + 2)y_1y_2^{m + 1} + (m + 1)y_2^{m + 2}}{(y_1 - y_2)^2}.
\end{align*}
The obtained result is equivalent to~\eqref{eq:hm_x1_x2_x2}.
\end{ex}

\begin{ex}
Let $n = 2$, $\ka=(3,2)$, and $m\in\bNz$.
In this case,
\begin{align*}
G_{(m),\ka}(y) &= 
\begin{bmatrix}
\binom{0}{0} & 0 & 0 & \binom{0}{0} & 0 \\[0.5ex]
\binom{1}{0}y_1 & \binom{1}{1} & 0 & \binom{1}{0}y_2 & \binom{1}{1}
\\[0.5ex]
\binom{2}{0}y_1^{2} & \binom{2}{1}y_1 & \binom{2}{2} & \binom{2}{0}y_2^{2} & \binom{2}{1}y_2
\\[0.5ex]
\binom{3}{0}y_1^{3} & \binom{3}{1}y_1^{2} & \binom{3}{2}y_1 & \binom{3}{0}y_2^{3} & \binom{3}{1}y_2^{2}
\\[0.5ex]
\binom{m + 4}{0}y_1^{m + 4} & \binom{m + 4}{1}y_1^{m + 3} & \binom{m + 4}{2}y_1^{m + 2} & \binom{m + 4}{0}y_2^{m + 4} & \binom{m + 4}{1}y_2^{m + 3}
\end{bmatrix}.
\end{align*}
In this example, the expansion of
the determinant is not so simple.
\end{ex}

\begin{rem}[computation complexity of~\eqref{eq:hom_rep_bialternant}]
\label{rem:hom_rep_bialternant_computational_complexity}
Suppose that all arithmetic operations are performed with objects of one type $T$,
where $T=\bQ$, or $T=\bQ(y_1,\ldots,y_n)$,
or $T$ is an approximate version of $\bR$ or $\bC$.
Unlike the general case~\eqref{eq:bialternant_formula_Schur_rep},
for $\la=(m)$,
high powers of $y_q$ appear only
in the last row of the matrix.
Using the exponentiation by squaring
and reasoning as in~\cite{GM2023},
we see that $\hom_m(y^{[\ka]})$
can be computed via~\eqref{eq:hom_rep_bialternant}
by using
\[
O(N^3)+O(N\log m)
\]
arithmetic operations in $T$.
For exact types like $T=\bQ$ or $T=\bQ(y_1,\ldots,y_n)$,
this analysis is not very useful,
because the cost of an arithmetic operation in $T$ depends on the ``size'' of the operands,
and the size of the objects usually grows
with every arithmetic operation.
On the other hand,
if $T$ is an approximate version of $\bR$ or $\bC$,
the arithmetic operations in $T$ are performed in a constant time,
$a^b$ is computed as $\exp(b\log a)$ for large $b$,
and $\exp$ and $\log$ are computed in a constant time,
then the computational cost
of~\eqref{eq:hom_rep_bialternant}
reduces to $O(N^3)$.
\end{rem}


\medskip

\section{\texorpdfstring{Definition of expressions $\boldsymbol{A}$ and $\boldsymbol{B}$}{Definition of the coefficients A and B}}
\label{sec:def_AB}

In this section,
we define expressions
$A_{y,\ka,s,r}$ and $B_{y,\ka,s,r}$
which play the main role
in Theorems~\ref{thm:decomposition_in_simple_fractions_B},
\ref{thm:decomposition_in_simple_fractions_A},
\ref{thm:hom_rep_via_powers_A},
\ref{thm:hom_rep_via_powers_B},
and
Corollary~\ref{cor:last_column_of_inverse_confluent_Vandermonde_matrix_S_version}.
We also establish a simple relation between them.

\begin{defn}
Let $\ka\in\bN^n$
and $y_1,\ldots,y_n$ be some variables or pairwise different complex numbers, $s\in \{1,\ldots, n\}$, and $r \in \{1,\ldots, \ka_s\}$.
For each $d$ in $\{1,\ldots,n\}\setminus\{s\}$, we put
\[
w_{s,d} \eqdef \frac{y_d}{y_d - y_s}.
\]
Furthermore, let
$w \eqdef
(w_{s,d})_{d\in\{1,\ldots,n\}\setminus\{s\}}$
and
$\la \eqdef
(\ka_d)_{d\in\{1,\ldots,n\}\setminus\{s\}}$.
Then, we define $A_{y, \ka, s, r}$
by the following formula:
\begin{equation}
\label{eq:A_def_via_hom}
A_{y,\ka,s,r}
\eqdef
\frac{(-y_s)^{|\ka| - \ka_s}}{\displaystyle\prod_{d\in\{1,\ldots,n\}\setminus\{s\}}(y_d -y_s)^{\ka_d}}
\hom_{\ka_s-r}\bigl( w^{[\la]} \bigr),
\end{equation}
i.e.,
\begin{equation}
\label{eq:A_def_via_hom_w}
A_{y,\ka,s,r}
=
\frac{(-y_s)^{|\ka| - \ka_s}}{\displaystyle\prod_{d\in\{1,\ldots,n\}\setminus\{s\}}(y_d -y_s)^{\ka_d}}
\,
\hom_{\ka_s-r}\left(w_{s,1}^{[\ka_1]},\ldots,w_{s,s-1}^{[\ka_{s-1}]}, w_{s,s+1}^{[\ka_{s + 1}]},\ldots, w_{s,n}^{[\ka_n]}\right).
\end{equation}
\end{defn}

Equivalently, using~\eqref{eq:hom_with_rep_combinatorial}, we can rewrite~\eqref{eq:A_def_via_hom} as
\begin{equation}
\label{eq:A_def_multisum}
A_{y,\ka,s,r}
= (-1)^{|\ka|-\ka_s}
\sum_{\substack{(k_1,\ldots,k_{s-1},k_{s+1},\ldots,k_n)\in\bNz^{n-1}\\
\sum_{j\neq s}k_j =\ka_s-r}}\;
\prod_{\substack{1\leq d\leq n \\ d \neq s}}
\binom{\ka_d +k_d-1}{k_d}
\frac{y_s^{\ka_d}y_d^{k_d}}%
{\left(y_d-y_s\right)^{\ka_d + k_d}}.
\end{equation}

\begin{defn}
\label{def:B_via_hom}
Let $\ka\in\bN^n$,
$y_1,\ldots,y_n$ be some variables or pairwise different numbers,
$s\in\{1,\ldots,n\}$,
and $r\in\{1,\ldots,\ka_s\}$.
For each $d$ in $\{1,\ldots,n\}\setminus\{s\}$,
we put
\[
z_{s,d} \eqdef \frac{1}{y_d - y_s}.
\]
Furthermore, let
$z \eqdef
(z_{s,d})_{d\in\{1,\ldots,n\}\setminus\{s\}}$
and
$\la \eqdef
(\ka_d)_{d\in\{1,\ldots,n\}\setminus\{s\}}$.
Then, we define $B_{y, \ka, s, r}$
by the following formula:
\begin{equation}
\label{eq:B_def_via_hom}
B_{y,\ka,s,r}
\eqdef
(-1)^{|\ka|-\ka_s}
\,
\Biggl(\,
\prod_{d\in\{1,\ldots,n\}\setminus\{s\}}
z_{s,d}^{\ka_d}
\Biggr)
\,
\hom_{\ka_s-r}(z^{[\la]}),
\end{equation}
i.e.,
\begin{equation}
\label{eq:B_def_via_hom_detailed}
B_{y,\ka,s,r}
\eqdef
(-1)^{|\ka|-\ka_s}
\Biggl(\,
\prod_{d\in\{1,\ldots,n\}\setminus\{s\}}
z_{s,d}^{\ka_d}
\Biggr)
\hom_{\ka_s-r}\left(z_{s,1}^{[\ka_1]},\ldots,z_{s, s-1}^{[\ka_{s-1}]}, z_{s, s+1}^{[\ka_{s + 1}]},\ldots, z_{s,n}^{[\ka_n]}\right),
\end{equation}
\end{defn}

Using~\eqref{eq:hom_with_rep_combinatorial}, we can rewrite~\eqref{eq:B_def_via_hom} as
\begin{equation}
\label{eq:B_def_multisum}
B_{y,\ka,s,r}
=
(-1)^{|\ka| - \ka_s} 
\sum_{\substack{(k_1,\ldots,k_{s-1},k_{s+1},\ldots,k_n)\in\bNz^{n-1}\\
\sum_{j\neq s}k_j =\ka_s-r}}\;
\prod_{\substack{1\leq d\leq n \\ d \neq s}} \binom{\ka_d + k_d - 1}{k_d} \frac{1}{(y_d-y_s)^{\ka_d + k_d}}.
\end{equation}

\begin{ex}
\label{ex:A_B_example}
    Let $n = 2$ and $\ka = (2,1)$. Then
\begin{align*}
A_{y, \ka, 1, 1} 
&=
 \frac{ - y_1}{y_2 - y_1}
\hom_{1}\left( \frac{y_2}{y_2 - y_1} \right) = -\frac{y_1 y_2}{(y_2 - y_1)^2},
\\
A_{y, \ka, 1, 2}
&=
\frac{-y_1}{y_2 - y_1}
\hom_{0}\left( \frac{y_2}{y_2 - y_1} \right) = -\frac{y_1}{y_2 - y_1},
\\
A_{y, \ka, 2,1}&= 
\frac{y_2^2}{(y_1 - y_2)^2}
\hom_{0}\left( \frac{y_1}{y_1 - y_2}, \frac{y_1}{y_1 - y_2} \right) = \frac{y_2^2}{(y_1 - y_2)^2},
\end{align*}
and
\begin{align*}
    B_{y,\ka, 1, 1} 
    &= 
    \frac{-1}{y_2 - y_1} \hom_{1}\left( \frac{1}{y_2 - y_1}\right)
    =
    -\frac{1}{(y_2 - y_1)^2}, \\
    B_{y, k, 1, 2} 
    &=
    \frac{-1}{y_2 - y_1} \hom_{0}\left( \frac{1}{y_2 - y_1}\right)
    = 
    - \frac{1}{y_2 - y_1}, \\
    B_{y, \ka, 2, 1} 
    &=
    \frac{1}{(y_1 - y_2)^{2}} \hom_{0}\left( \frac{1}{y_1 - y_2}\right) = \frac{1}{(y_1 - y_2)^2}.
\end{align*}

\end{ex}

\begin{prop}
\label{prop:A_via_B_trivial}
Let $\ka\in\bN^n$,
$s\in\{1,\ldots,n\}$,
and $r\in\{1,\ldots,\ka_s\}$.
Then the following identity holds
in $\bF(y_1,\ldots,y_n)$:
\begin{equation}
\label{eq:A_via_B_simple}
A_{y,\ka,s,r}
=
(-1)^{|\ka| - r}
y_s^{r-\ka_s}
\left(\prod_{d\in\{1,\ldots,n\}\setminus\{s\}} y_d^{-\ka_d}\right)
\,
B_{1/y,\ka,s,r}.
\end{equation}
\end{prop}

\begin{proof}
First, we notice that
\[
\frac{1}{\frac{1}{y_d} - \frac{1}{y_s}}
= \frac{(-y_s)y_d}{y_d - y_s} = -y_s w_{s,d}.
\]
We substitute these expressions instead of $z_{s,d}$
in~\eqref{eq:B_def_via_hom},
then we apply the homogeneous property of $\hom_{\ka_s-r}$:
\begin{align*}
B_{1/y,\ka,s,r}
&=
(-1)^{|\ka|-\ka_s}
\left(
\prod_{d\in\{1,\ldots,n\}\setminus\{s\}} (-y_s w_{s,d})^{\ka_d}
\right)
\,
\hom_{\ka_s-r}(-y_s w^{[\la]})
\\[1ex]
&=
(-1)^{\ka_s-r}\,
y_s^{|\ka|-r}\,
\left(\prod_{d\in\{1,\ldots,n\}\setminus\{s\}}\,\frac{y_d^{\ka_d}}{(y_d-y_s)^{\ka_d}}\right)\,
\hom_{\ka_s-r}(w^{[\la]}).
\end{align*}
After substituting this expression
into the right-hand side of~\eqref{eq:A_via_B_simple}
and simplifying,
we obtain the right-hand side of~\eqref{eq:A_def_via_hom}.
\end{proof}

\medskip

\section{Decomposition of rational functions with unital numerator}
\label{sec:rational_decomposition}

The main results of this section are Theorems~\ref{thm:decomposition_in_simple_fractions_B} and~\ref{thm:decomposition_in_simple_fractions_A}.

Recall the derivative of a polynomial is defined in the formal way, see Remark~\ref{rem:derivative_of_polynomial}. The Leibniz rule holds for the derivative
of the product of two polynomials:
if $f,g\in\bF[t]$, then $(fg)'=f'g+fg'$;
see, e.g., 
\cite[Chapter IV, Proposition 3.5]{Lang2005}.
This rule extends to rational functions
($f,g\in\bF(t)$).
Using the mathematical induction,
it is easy to prove the general Leibniz rule for the $m$th derivative of the product of $n$ rational functions. 
Namely, for every $f_1,\ldots,f_n$ in $\bF(t)$
and every $m$ in $\bNz$,
the following identity holds in $\bF(t)$:
\begin{equation}
\label{eq:Leibniz_rule_deriv_product}
\bigl(f_1\cdots f_n\bigr)^{(m)}(t)
=\sum_{k\in\cM_{n,m}}
\frac{m!}{k_1!\cdots k_n!}
\prod_{j=1}^n f_j^{(k_j)}(t).
\end{equation}

\begin{lem}
\label{lem:derivative_of_special_product}
Let $f \in \bF(t)$,
$t_0 \in \bF$ such that $t_0$ is not a pole of $f$,
$p\in\bN$, and
\[
g(t) \eqdef (t-t_0)^p f(t).
\]
Then
\[
g^{(m)}(t_0)
=
\begin{cases}
0, & \text{if} \quad 0\le m<p,\\
m!\,R(t_0), & \text{if} \quad m=p.
\end{cases}
\]
\end{lem}

\begin{proof}
Let $m\le p$.
By~\eqref{eq:Leibniz_rule_deriv_product},
\[
g^{(m)}(t)
=
\sum_{k=0}^{m}
\binom{m}{k}
f^{(k)}(t)
\frac{p!}{(p+k-m)!}(t-t_0)^{p+k-m}.
\]
Substitute $t_0$ instead of $t$:
\[
g^{(m)}(t_0)
=
\sum_{k=0}^{m}
\binom{m}{k}
f^{(k)}(t_0)
\frac{p!}{(p+k-m)!}\ 0^{p+k-m}.
\]
If $m<p$, all terms on the right-hand side are zero.
If $m=p$, all terms with $k>0$ are zero, and the term corresponding to $k=0$ equals $m!\, R(t_0)$.
\end{proof}

\begin{lem}
\label{lem: derivative_partial_fraction}
Let $a\in \bF$, $p\in\bN$,
and let $q$ be the following rational function:
\[
q(t) = \frac{1}{(t - a)^p}.
\]
Then for every $m$ in $\bNz$,
the following identity holds in $\bF(t)$:
\[
q^{(m)}(t)
= \frac{(-1)^{m}\,(p + m-1)!}{(p-1)!\,(t-a)^{m + p}}.
\]
\end{lem}

\begin{proof}
This result is well known and can be proven by induction.
\end{proof}

\begin{thm}[partial fraction decomposition of rational functions with unital numerator]
\label{thm:decomposition_in_simple_fractions_B}
Let $\ka\in\bN^n$.
Then the following identity holds in
$\bF(y_1,\ldots,y_n,t)$:
\begin{equation}
\label{eq:decomposition_in_simple_fractions_B}
 \frac{1}{\displaystyle \prod_{j=1}^{n}(t-y_j)^{\ka_j}} = \sum_{s=1}^{n} \sum_{r=1}^{\ka_s}\frac{B_{y,\ka,s,r}}{(t-y_s)^{r}}.
\end{equation}
\end{thm}

\begin{proof}
Let $f(t)$ be the left-hand side of~\eqref{eq:decomposition_in_simple_fractions_B}.
It is a well-known fact
(see~\cite[Chapter~4, Theorem~4.2]{Lang2005}) that $f$ admits a unique decomposition of the form
\begin{equation}
\label{eq:indefinite_decomposition_in_simple_fractions}
f(t)
= \sum_{s=1}^{n} \sum_{r=1}^{\ka_s}
\frac{c_{s,r}}{(t-y_s)^{r}},
\end{equation}
where $c_{s,r}$ are some coefficients depending on $y$, $\ka$, $s$, and $r$.
Our objective is to find these coefficients.
In the rest of the proof,
we fix $s$ in $\{1,\ldots,n\}$.
Let
\begin{equation}
\label{eq:Gs_lhs}
F_s(t)
\eqdef (t-y_s)^{\ka_s} f(t)
= \prod_{j\in\{1,\ldots,n\}\setminus\{s\}}
\frac{1}{(t-y_j)^{\ka_j}}.
\end{equation}
Due to~\eqref{eq:indefinite_decomposition_in_simple_fractions},
\begin{align*}
F_s(t)
= \sum_{j=1}^{\ka_s}
c_{s,j}(t-y_s)^{\ka_s-j}
+ (t-y_s)^{\ka_s}H_s(t),\qquad
H_s(t)\eqdef
\sum_{\substack{d\in \{1,\ldots,n\}\\ d \neq s}} \; \sum_{r=1}^{\ka_d}
\frac{c_{d,r}}{(t-y_d)^r}.
\end{align*}
Given $m$ with $0\le m\le \ka_s-1$,
we compute $F_s^{(m)}(y_s)$ using Lemma~\ref{lem:derivative_of_special_product}:
\[
F_s^{(m)}(y_s)
= m!\,c_{s,\ka_s-m}.
\]
On the other hand, by~\eqref{eq:Gs_lhs}, \eqref{eq:Leibniz_rule_deriv_product}, and Lemma~\ref{lem: derivative_partial_fraction},
\[
F_s^{(m)}(y_s)
= \sum_{\sum_{j\neq s}k_j=m} \;
\frac{m!}{\prod_{\substack{1\le d\le n\\d\ne s}}k_d!} \;
\prod_{\substack{1\le d\le n \\ d \ne s}}
\frac{(-1)^{k_d}\,(\ka_d + k_d-1)!}{(\ka_d-1)!\,(y_s-y_d)^{\ka_d + k_d}}.
\]
Comparing to~\eqref{eq:B_def_multisum}
we see that
\begin{equation}
\label{eq:Gs_deriv}
F_s^{(m)}(y_s)
=m!\,B_{y,\ka,s,\ka_s-m}.
\end{equation}
Now we substitute $m=\ka_s-r$, where $1\le r\le \ka_s$,
and conclude that $c_{s,r}=B_{y,\ka,s,r}$.
\end{proof}

The following result is closely related to
Theorem~\ref{thm:decomposition_in_simple_fractions_B};
it is contained in the proof above.

\begin{prop}
\label{prop: derivative_F_in_terms_of_B}
Let $\ka\in\bN^n$
and $s\in\{1,\ldots,n\}$.
Denote by $F_s$ the rational function
\[
F_s(t)
\eqdef
\prod_{j\in\{1,\ldots,n\}\setminus\{s\}}
\frac{1}{(t-y_j)^{\ka_j}}.
\]
Then for every $m$ in $\{1,\ldots,\ka_s\}$,
the following identity holds in
$\bF(y_1,\ldots,y_n)$:
\begin{equation}\label{eq:special_fraction_derivative}
F_s^{(m)}(y_s) = m!\;B_{y,\ka,s,\ka_s-m}.
\end{equation}
\end{prop}

\begin{thm}
\label{thm:decomposition_in_simple_fractions_A}
Let $\ka\in\bN^n$.
The following identity holds in
$\bF(y_1,\ldots,y_n,t)$:
\begin{equation}
\label{eq:decomposition_in_simple_fractions_A}
\frac{1}{\displaystyle \prod_{j=1}^{n}(1-y_j t)^{\ka_j}} = \sum_{s=1}^{n} \sum_{r=1}^{\ka_s}\frac{A_{y, \ka,s,r}}{(1-y_s t)^{r}}.
\end{equation}
\end{thm}

\begin{proof}
It follows from Theorem~\ref{thm:decomposition_in_simple_fractions_B} applied to
$1/y_1,\ldots,1/y_n$
instead of $y_1,\ldots,y_n$,
by using the identity
\[
\frac{1}{t - \frac{1}{y_j}}
= - \frac{y_j}{1 - y_jt}
\]
and Proposition~\ref{prop:A_via_B_trivial}.
\end{proof}

\begin{cor}
\label{cor:decomposition1_simple_poles_in_simple_fractions}
Then following identity holds in $\bF(y_1,\ldots,y_n,t)$:
\begin{equation}
\label{eq:decomposition1_simple_poles_in_simple_fractions}
\frac{1}{\displaystyle \prod_{j=1}^{n}(1-y_j t)}
= \sum_{s=1}^{n}
\frac{y_s^{n-1}}{\displaystyle\prod_{j\in\{1,\ldots,n\}\setminus\{s\}}\,
\frac{1}{1-y_s t}}.
\end{equation}
\end{cor}

\begin{proof}
This is a particular case of Theorem~\ref{thm:decomposition_in_simple_fractions_A},
with $\ka_1=\cdots=\ka_n=1$.
This particular case is well known.
\end{proof}

\medskip

\section{\texorpdfstring{Expansions of $\boldsymbol{\hom_m(y^{[\ka]})}$}{Expansions of hm(ykappa)}}
\label{sec:hom_rep_expansion}

The main results of this section are Theorems~\ref{thm:hom_rep_via_powers_A} and~\ref{thm:hom_rep_via_powers_B}.

In the proof of Theorem~\ref{thm:hom_rep_via_powers_A},
we will use the following well-known formula for the sum of the geometric progression~\cite[Equation 5.56]{GrahamKnuthPatashnik1994}:
\begin{equation}
\label{eq:negative_power_expansion}
\frac{1}{(1-y_s t)^{r}}
=
\sum_{m=0}^\infty
\binom{r + m -1}{m}
y_s^m t^m.
\end{equation}

\begin{thm}
\label{thm:hom_rep_via_powers_A}
Let $\ka\in\bN^n$.
Then the following identity holds
in $\bF(y_1,\ldots,y_n)$:
\begin{equation}
\label{eq:hom_rep_via_powers_A}
\hom_m(y^{[\ka]})
=
\sum_{s=1}^{n} \sum_{r=1}^{\ka_s}
A_{y,\ka,s,r}
\binom{m + r - 1}{r - 1} y_s^m.
\end{equation}
As a consecuence,
\eqref{eq:hom_rep_via_powers_A}
also holds if $y_1,\ldots,y_n$
are some pairwise different elements of $\bF$.
\end{thm}

\begin{proof}
Let $R\eqdef\bF(y_1,\ldots,y_n)$.
We are going to deal with elements of $R[[t]]$
(formal series)
and elements of $R(t)$
which we identify with $\bF(y_1,\ldots,y_n,t)$
(rational functions).
Notice that $R[t]$
is naturally embedded in both $R[[t]]$ and $R(t)$.
To justify the passes from formal series to rational functions and vice versa,
we prefer to multiply them by the following polynomial:
\[
\Om_\ka(y,t)
\eqdef
\prod_{j=1}^n (1-y_j t)^{\ka_j}.
\]
Thanks to this factor, all sides
in the following chains of equalities
simplify to elements of $R[t]$.

We start with the generating function $\Hom(x)(t)$ with $y^{[\ka]}$ instead of $x$,
multiply it by $\Om_\ka(y,t)$,
and use~\eqref{eq:Homgen_explicit}:
\[
\Om_\ka(y,t)\;
\sum_{m=0}^\infty \hom_m(y^{[\ka]})t^m
=
\Om_\ka(y,t)\;
\Hom(y^{[\ka]})(t)
=1.
\]
Next, we apply the partial fraction decomposition from Theorem~\ref{thm:decomposition_in_simple_fractions_A}:
\[
1
=
\Om_\ka(y,t)\;
\sum_{s=1}^n
\sum_{r=1}^{\ka_s}
\frac{A_{y, \ka, s,r}}{(1-y_st)^{r}}
=
\sum_{s=1}^n
\sum_{r=1}^{\ka_s}
\frac{A_{y, \ka, s,r} \Om_\ka(y,t)}{(1-y_st)^{r}}.
\]
All summands in the last expression are elements of $R[t]$.
Using~\eqref{eq:negative_power_expansion},
we represent each summand via a formal series,
and after that we sum the obtained formal series:
\begin{align*}
\sum_{s=1}^n
\sum_{r=1}^{\ka_s}
\frac{A_{y, \ka, s,r} \Om_\ka(y,t)}{(1-y_st)^{r}}
&=
\sum_{s=1}^n
\sum_{r=1}^{\ka_s}
A_{y, \ka, s,r}
\Om_\ka(y,t)
\sum_{m=0}^\infty
\binom{r + m - 1}{m}y_s^m t^m
\\
&=
\Om_\ka(y,t)
\sum_{m=0}^\infty
\left( \sum_{s=1}^{n}\sum_{r=1}^{\ka_s}
A_{y, \ka, s,r}\binom{r + m - 1}{m}y_s^m\right)
t^m.
\end{align*}
Thereby, we have obtained the following identity in $R[[t]]$:
\[
\Om_\ka(y,t)
\sum_{k=0}^\infty \hom_k(y^{[\ka]})t^k
=
\Om_\ka(y,t)
\sum_{m=0}^\infty
\left( \sum_{s=1}^{n}\sum_{r=1}^{\ka_s}
A_{y, \ka, s,r}\binom{r + m - 1}{m}y_s^m\right)t^m.
\]
Since $R[[t]]$ is an integral domain,
the factor $\Om_\ka(y,t)$ can be canceled.
Finally, equating the coefficients of the formal series,
we get~\eqref{eq:hom_rep_via_powers_A}.
\end{proof}

We will derive another formula for $\hom_m(y^{[\ka]})$ using Corollary~\ref{cor:hom_rep_bialternant} and expanding the determinant of $G_{(m),\ka}(y)$ along the last row. We need the last column of the adjugate matrix of $G_{(m),\ka}(y)$.

\begin{prop}
\label{prop:adj_Gm_last_row}
Let $\ka\in\bN^n$, $N = |\ka|$, $m \in \bNz$, $k\in \{1,\ldots,N\}$.
Then the following identity holds in $\bF(y_1,\ldots,y_n)$:
\begin{equation}
\label{eq:adj_Gm_last_row}
\bigl(\adj G_{(m),\ka}(y)\bigr)_{k,N}
=
\bigl(V_\ka(y)^{-1}\bigr)_{k,N} \;
\det V_\ka(y).
\end{equation}
\end{prop}

\begin{proof}
By Remark~\ref{rem:G_m}, the first $m-1$ rows of $G_{(m),\ka}(y)$ coincide with the corresponding rows of $V_\ka(y)$.
Therefore, the $(N,k)$ cofactor of $G_{(m),\ka}(y)$ is equal to the $(N,k)$ cofactor of $V_\ka(y)$. Equivalently, $\adj G_{(m),\ka}(y)$ and $\adj V_\ka(y)$ have the same last column.
Finally, the adjugate of $V_\ka(y)$ is the inverse of $V_\ka(y)$ multiplied by $\det V_\ka(y)$. 
\end{proof}

Mou\c{c}ouf and Zriaa~\cite[Theorem~2.2]{MoucoufZriaa2022}
gave a formula for the inverse of the confluent Vandermonde matrix
(see also algorithms and historial information in~\cite{HouPang2002,Respondek2011,Respondek2024}).
We will rewrite their result in a slightly different notation.

\begin{thm}[Mou\c{c}ouf and Zriaa, 2022]
\label{thm:Moucouf_inverse_matrix_confluent}
Let $\ka\in\bN^n$
and $y_1,\ldots,y_n$ be some variables.
Then $V_{\ka}(y)^{-1}$
has the following block structure: 
\[
V_{\ka}(y)^{-1}
=
\begin{bmatrix}
\mathcal{L}_{1}\\
\mathcal{L}_{2}\\
\vdots\\
\mathcal{L}_{n}
\end{bmatrix},
\]
where for every $s$ in $\{1,\ldots, n\}$,
the $\ka_s \times N$-matrix $\mathcal{L}_{s}$ is given by
\[
\mathcal{L}_{s} 
\eqdef
\left[
\frac{1}{(j - 1)!}L_{s,r - 1}^{(j - 1)} (0)
\right]_{1\leq r \leq \ka_s, 1\leq j \leq N},
\]
where, for each $p$ in $\{0, \ldots, \ka_s\}$,
\begin{equation}
\label{eq:def_L_function_Moucouf_Zriaa}
L_{s,p}(t)
\eqdef
\frac{(t - y_s)^{p}}{
F_s(t)}
\sum_{m = 0}^{\ka_s - p - 1}\frac{F_s^{(m)}(y_s)}{m!} \; (t - y_s)^{m}, 
\end{equation}
i.e.,
\begin{equation}
\label{eq:def_L_function_Moucouf_Zriaa2}
L_{s,p}(t)
=
\left(\prod_{\substack{1\leq q \leq n \\ q\neq s}} (t - y_q)^{\ka_q}\right)
(t - y_s)^{p}
\sum_{m = 0}^{\ka_s - p - 1}\frac{F_s^{(m)}(y_s)}{m!} \; (t - y_s)^{m}.
\end{equation}
\end{thm}

It turns out that the last column of the inverse of  $V_\ka(y)$ can be expressed in terms of $B_{y, \ka, s , r}$
and therefore in terms of complete homogeneous polynomials.

\begin{cor}[last column of the inverse of the confluent Vandermonde matrix]
\label{cor:last_column_of_inverse_confluent_Vandermonde_matrix_S_version}
Let $\ka\in\bN^n$.
Then for every $s$ in $\{1, \ldots, n\}$
and every $r$ in $\{1, \ldots, \ka_s\}$,
the following identity holds in $\bF(y_1,\ldots,y_n)$:
\[
\left(V_{\ka}(y)^{-1}\right)_{\si(\ka)_{s - 1} + r, N}
=
B_{y,\ka,s,r}.
\]
\end{cor}

\begin{proof} 
Let $k = \rho_\ka(s,r) = \sigma(\ka)_{s - 1} + r$. According to the constructions of Section~\ref{sec:repeating},
$k$ is the ``global index'' of the component $r$
in block $s$. 
So, by Theorem~\ref{thm:Moucouf_inverse_matrix_confluent}, the component $(k,N)$ of $V_\ka(y)^{-1}$ is the component $(r, N)$ of $\mathcal{L}_s$, which can be computed via the $(N-1)$th derivative of $L_{s,r-1}$: 
\[
\left(V_{\ka}(y)^{-1}\right)_{k, N} 
=
(\mathcal{L}_s)_{r,N}
=
\frac{1}{(N - 1)!}L_{s,r - 1}^{(N - 1)} (0).
\]
By~\eqref{eq:def_L_function_Moucouf_Zriaa2}, $L_{s,r-1}$ is a polynomial of degree $N-1$, and the coefficient of $t^{N-1}$ is $\frac{1}{(\ka_s-r)!}F_s^{(\ka_s - r)}(y_s)$. Therefore,
\[
\left(V_\ka(y)^{-1}\right)_{k, N}  = \frac{1}{(\ka_s-r)!}F_s^{(\ka_s - r)}(y_s).
\]
Finally, by Proposition~\ref{prop: derivative_F_in_terms_of_B}, the latter expression is $B_{y,\ka,s,r}$.
\end{proof}

\begin{thm}
\label{thm:hom_rep_via_powers_B}
Let $\ka\in\bN^n$ and $N=|\ka|$.
Then the following identity holds
in $\bF(y_1,\ldots,y_n)$:
\begin{equation}
\label{eq:hom_rep_via_powers_B}
\hom_m(y^{[\ka]})
=
\sum_{s=1}^{n}
\sum_{r=1}^{\ka_s} 
B_{y,\ka,s,r}\;
\binom{m + N - 1}{r - 1}\;
y_s^{m + N - r}.
\end{equation}
As a consequence,
\eqref{eq:hom_rep_via_powers_B} also holds
if $y_1,\ldots,y_n$
are pairwise different elements of $\bF$.
\end{thm}

\begin{proof}
We express $\hom_m(y^{[\ka]})$
as a quotient of two determinants
by Corollary~\ref{cor:hom_rep_bialternant},
expand the determinant
of $G_{(m),\ka}(y)$ by cofactors
along the last row,
and apply Proposition~\ref{prop:adj_Gm_last_row}:
\begin{align*}
\hom_m(y^{[\ka]})
&=
\frac{\det G_{(m),\ka}(y)}{\det V_\ka(y)}
=
\frac{1}{\det V_\ka(y)}
\sum_{k=1}^N
G_{(m),\ka}(y)_{N,k}
\bigl(\adj G_{(m),\ka}(y)\bigr)_{k,N}
\\
&=
\sum_{k=1}^N
G_{(m),\ka}(y)_{N,k}
\bigl(V_\ka(y)^{-1}\bigr)_{k,N}.
\end{align*}
Next, we use the change of variables $\rho_\ka$
from Section~\ref{sec:repeating}:
\[
\hom_m(y^{[\ka]})
=
\sum_{s=1}^{n}
\sum_{r=1}^{\ka_s} 
G_{(m),\ka}(y)_{N,\si(\ka)_{s-1}+r}
\bigl(\adj G_{(m),\ka}(y)\bigr)_{\si(\ka)_{s-1}+r,N}.
\]
Finally, by formulas from
Remark~\ref{rem:G_m}
and Corollary~\ref{cor:last_column_of_inverse_confluent_Vandermonde_matrix_S_version},
we get~\eqref{eq:hom_rep_via_powers_B}.
\end{proof}

Both~\eqref{eq:hom_rep_via_powers_A}
and~\eqref{eq:hom_rep_via_powers_B}
can be written in the form~\eqref{eq:general_form_of_expansion}.
We will directly prove that~\eqref{eq:hom_rep_via_powers_A}
is equivalent to~\eqref{eq:hom_rep_via_powers_B},
but first we illustrate these formulas by a simple example.

\begin{ex}
\label{ex:hom_21}
Let $n = 2$ and $\ka = (2,1)$.
We use explicit expresions for $A_{y,\ka,s,r}$ and $B_{y,\ka,s,r}$ from Example~\ref{ex:A_B_example}. 
First, we will compute $\hom_m(y^{[\ka]})$ by \eqref{eq:hom_rep_via_powers_A}:
\begin{align}
\label{eq:example_P1_A}
P_{y,\ka,1}(m) &= A_{y,\ka,1,1} \binom{m}{0} + A_{y,\ka,1,2}\binom{m + 1}{1} =  -\frac{y_1 y_2}{(y_2 - y_1)^2} - \frac{ (m + 1)y_1}{y_2 - y_1},\\
\label{eq:example_P2_A}
P_{y,\ka,2}(m) &= A_{y, \ka, 2,1}\binom{m}{0} 
= 
\frac{y_2^2}{(y_1 - y_2)^2}.
\end{align}
So,
\begin{align*}
\hom_m\bigl(y_1^{[2]}, y_2^{[1]} \bigr)
&=
P_{y,\ka,1}(m)\,y_1^m
+P_{y,\ka,2}(m)\,y_2^m \\
&= 
\left(-\frac{y_1 y_2}{(y_2 - y_1)^2} - \frac{ (m + 1)y_1}{y_2 - y_1}\right)\,y_1^m
+
\frac{y_2^2}{(y_1 - y_2)^2}\,y_2^m.
\end{align*}
Now, we will compute $\hom_m(y^{[\ka]})$
by~\eqref{eq:hom_rep_via_powers_B}:
\begin{align}
\label{eq:example_P1_B}
P_{y,\ka,1}(m) &= B_{y,\ka,1,1} \binom{m + 2}{0} y_1^{2} + B_{y,\ka,1,2}\binom{m + 2}{1}y_1 =  -\frac{y_1^2}{(y_2 - y_1)^2} - \frac{ (m + 2)y_1}{y_2 - y_1},\\
\label{eq:example_P2_B}
P_{y,\ka,2}(m) &= B_{y, \ka, 2,1}\binom{m + 2}{0}y_2^2 
= 
\frac{y_2^2}{(y_1 - y_2)^2},
\end{align}

\begin{align*}
\hom_m\bigl(y_1^{[2]}, y_2^{[1]} \bigr)
&=
P_{y,\ka,1}(m)\,y_1^m
+P_{y,\ka,2}(m)\,y_2^m \\
&= 
\left(-\frac{y_1^2}{(y_2 - y_1)^2} - \frac{ (m + 2)y_1}{y_2 - y_1}\right)\,y_1^m
+
\frac{y_2^2}{(y_1 - y_2)^2}\,y_2^m.
\end{align*}
The right-hand sides of
\eqref{eq:example_P2_A}
and~\eqref{eq:example_P2_B} coincide.
Let us show that the right-hand sides of~\eqref{eq:example_P1_B} and~\eqref{eq:example_P1_A}
are equal:
\begin{align*}
-\frac{y_1^2}{(y_2 - y_1)^2} - \frac{ (m + 2)y_1}{y_2 - y_1}
&=
-\frac{y_1^2}{(y_2 - y_1)^2} - \frac{y_1}{y_2 - y_1} - \frac{(m + 1)y_1}{y_2 - y_1}
\\
&=
-\frac{y_1^2}{(y_2 - y_1)^2} - \frac{y_1(y_2 - y_1)}{(y_2 - y_1)^2} - \frac{(m + 1)y_1}{y_2 - y_1}\\
&=
-\frac{y_1y_1}{(y_2 - y_1)^2}  - \frac{(m + 1)y_1}{y_2 - y_1}.
\end{align*}
So, in this example,
\eqref{eq:hom_rep_via_powers_A} and \eqref{eq:hom_rep_via_powers_B} yield the same answer.
\hfill\qed
\end{ex}

Now we will consider a general situation
and show directly that the right-hand sides of \eqref{eq:hom_rep_via_powers_A} and \eqref{eq:hom_rep_via_powers_B} coincide. 
We use a well-known formula for the sum of products of binomial coefficients~\cite[Equation 5.26]{GrahamKnuthPatashnik1994}:
\begin{equation}
\label{eq: sum_Graham_Knuth_Patashnik}
\sum_{j = 0}^l \binom{l - j}{m}\binom{q + j}{n} = \binom{l + q + 1}{m + n + 1}.
\end{equation}

\begin{prop}
Let $\ka\in\bN^n$.
Then the following identity holds in $F(y_1,\ldots,y_n)$: 
\begin{equation}
\label{eq: equivalence_A_and_B}
\sum_{r=1}^{\ka_s}
\binom{m + r - 1}{r - 1}
A_{y,\ka,s,r}
=
\sum_{r=1}^{\ka_s} 
\binom{m + N - 1}{r - 1}\;
B_{y,\ka,s,r}\;y_s^{N -r}.
\end{equation}
\end{prop}

\begin{proof}
First, we establish a connection between the auxiliar variables $w_{s,d}$ and $z_{s,d}$ appearing in \eqref{eq:B_def_via_hom} and \eqref{eq:A_def_via_hom}:
\[
w_{s,d}
= \frac{y_d}{y_d - y_s}
= 1 + \frac{y_s}{y_d - y_s} = 1 + y_s z_{s,d}.
\]
Furthermore, we use the following short notation for the list of multiplicities and variables appearing in~\eqref{eq:A_def_via_hom} and \eqref{eq:B_def_via_hom}:
\[
\la = \left(\ka_d \right)_{d \in \{1,\ldots, n\} \setminus \{s\}} = (\ka_1,\ldots,\ka_{s-1}, \ka_{s + 1}, \ldots, \ka_n),
\]
\[
w = (w_{s,1},\ldots, w_{s,s-1}, w_{s, s+ 1}, \ldots,w_{s,n}),\qquad
z = (z_{s,1},\ldots,z_{s,s-1}, z_{s, s+ 1}, \ldots,z_{s, n}).
\]
With this notation, $w^{[\la]} = 1^{[N - \ka_s]} + y_s z^{[\la]}$. By Proposition~\ref{prop:Gomezllata}, 
\begin{align*}
A_{y,\ka,s,r}
&=
(-y_s)^{N - \ka_s}\left(\prod_{d \neq s}z_{s,d}^{\ka_d}\right)
\hom_{\ka_s-r}\left( w^{[\la]}\right)
\\[1ex]
&=
(-y_s)^{N - \ka_s}\left(\prod_{d \neq s}z_{s,d}^{\ka_d}\right)
\hom_{\ka_s-r}\left(1^{[N -\ka_s]} + y_sz^{[\la]} \right)
\\[1ex]
&= 
(-y_s)^{N - \ka_s}\left(\prod_{d \neq s}z_{s,d}^{\ka_d}\right)
\sum_{k = 0}^{\ka_s - r}
\binom{N - r - 1}{\ka_s - r - k}
\hom_{k}\left(y_sz^{[\la]}\right)
\\[1ex]
&= 
(-1)^{N - \ka_s}\left(\prod_{d \neq s}z_{s,d}^{\ka_d}\right) \sum_{k = 0}^{\ka_s - r}y_s^{N - \ka_s + k}
\binom{ N  - r - 1}{\ka_s - r - k}
\hom_{k}\left(z^{[\la]}\right).
\end{align*}
We denote by $L$ the left-hand side of~\eqref{eq: equivalence_A_and_B}, substitute the obtained expression for $A_{y,\ka,s,r}$, and change the order of the sums:
\begin{align*}
L 
&=\sum_{r=1}^{\ka_s}
\binom{m + r - 1}{r - 1} (-1)^{N - \ka_s}\left(\prod_{d \neq s}z_{s,d}^{\ka_d}\right)
\sum_{k = 0}^{\ka_s - r}
y_s^{N - \ka_s + k}
\binom{N - r - 1}{\ka_s - r - k}
\hom_{k}\left(z^{[\la]}\right)\\[1ex]
&=
\sum_{k = 0}^{\ka_s - 1} 
(-1)^{N - \ka_s}\left(\prod_{d \neq s}z_{s,d}^{\ka_d}\right)y_s^{N - \ka_s + k} \hom_{k}\left(z^{[\la]}\right)
\sum_{r = 1}^{\ka_s - k} \binom{m + r - 1}{r - 1} \binom{N - r - 1}{\ka_s - r - k}.
\end{align*}
The inner sum can be extended with zero summands and simplified with help of \eqref{eq: sum_Graham_Knuth_Patashnik}:
\begin{align*}
&\sum_{r = 1}^{\ka_s - k} \binom{m + r - 1}{r - 1} \binom{N - r - 1}{\ka_s - r - k} =
\sum_{r = 1}^{N-1} \binom{m + r - 1}{r - 1} \binom{N - r - 1}{\ka_s - r - k} \\
&\qquad= \sum_{j = 0}^{N - 2} 
\binom{N - 2 - j}{N-\ka_s +k - 1}  \binom{m + j}{m}
= 
\binom{N + m - 1}{\ka_s -k - 1}.    
\end{align*}
Finally, we substitute the last expression instead of the inner sum and make the change of variables $r = \ka_s - k$:
\begin{align*}
L&=
\sum_{k = 0}^{\ka_s - 1} 
(-1)^{N - \ka_s}\left(\prod_{d \neq s}z_{s,d}^{\ka_d}\right)y_s^{N - \ka_s + k} \hom_{k}\left(z^{[\la]}\right)
\binom{N + m - 1}{\ka_s -k - 1}\\
&=
\sum_{r = 1}^{\ka_s} 
(-1)^{N - \ka_s}\left(\prod_{d \neq s}z_{s,d}^{\ka_d}\right) y_s^{N - r}\hom_{\ka_s - r}\left(z^{[\la]}\right) 
\binom{N + m - 1}{r - 1}\\
&=
\sum_{r=1}^{\ka_s}
\binom{N + m - 1}{r - 1}
B_{y,\ka,s,r} y_s^{N-r}.
\end{align*}

\end{proof}

\begin{rem}[computational complexity of~\eqref{eq:hom_rep_via_powers_B} and~\eqref{eq:hom_rep_via_powers_A}]
Suppose that $T$ is an approximative version of $\bR$ or $\bC$,
and $y_1,\ldots,y_n$ are pairwise different elements of $T$.
Each coefficient $A_{y,\ka,s,r}$ or $B_{y,\ka,s,r}$
can be computed by~\eqref{eq:hom_rep_bialternant}
using $O(N^3)$ operations in $T$.
So, the total number of operations
can be estimated by $O(N^4)+O(N\log m)$.
It is usually bigger than the number of operations used in~\eqref{eq:hom_rep_bialternant}.
Nevertheless,~\eqref{eq:hom_rep_via_powers_B} and~\eqref{eq:hom_rep_via_powers_A} can be more useful than~\eqref{eq:hom_rep_bialternant}
when computing $\hom_m(y^{[\ka]})$
for many values of $m$,
with fixed $y$ and $\ka$.
\end{rem}

\begin{rem}
It is easy to prove by mathematical induction that
\begin{equation}
\label{eq:binomial_and_stirling}
\binom{x+n}{n}
=
\frac{1}{n!}
\sum_{j=0}^n
\stirlingone{n+1}{j+1}
\,
x^j,
\end{equation}
where
$\stirlingone{a}{b}$
are unsigned Stirling numbers of the first kind.
Expanding
$\binom{m+r-1}{r-1}$
in~\eqref{eq:hom_rep_via_powers_A}
we get
\begin{equation}
\label{eq:hom_via_C}
\hom_m(y^{[\ka]})
=
\sum_{n=1}^n
\sum_{j=0}^{\ka_s-1}
C_{y,\ka,s,j} m^j y_s^m,
\end{equation}
where
\begin{equation}
\label{eq:C_via_A}
C_{y,\ka,s,j}
=
\sum_{r = j + 1}^{\ka_s}
\frac{1}{(r-1)!}
\stirlingone{r}{j + 1}
A_{y,\ka,s,r}.
\end{equation}
It would be interesting to obtain
a simpler expression for $C_{y,\ka,s,j}$.
\end{rem}


\begin{rem}
We have tested the main formulas of this paper
in Sagemath~\cite{Sagemath}.
Our programs are freely available at~\cite{GM2024testshomrep}.
In particular,
we have computed $\hom_m(y^{[\ka]})$
using symbolic computations over
$\bQ(y_1,\ldots,y_n)$,
by~\eqref{eq:hom_rep_via_powers_A},
\eqref{eq:hom_rep_via_powers_B},
and~\eqref{eq:hom_rep_bialternant},
for all $n$, $m$, $\ka$
with
$n\le 8$, $m\le 10$, and $|\ka|\le 10$.
For all these values of parameters,
the polynomials obtained by different formulas coincide.
In the same spirit,
we have tested many other formulas of this paper.
\end{rem}

\label{endlabel}
\end{document}